\documentclass[reqno]{amsart}
\usepackage{amssymb, amsmath, amsthm, amscd, mathrsfs}

\usepackage{paralist}
\usepackage{cite}
\usepackage{bigints}
\usepackage{dsfont}
\usepackage{empheq}
\usepackage{amssymb}
\usepackage{cases}
\usepackage{enumitem}
\allowdisplaybreaks

\usepackage{caption}
\usepackage{booktabs}

\usepackage{float}
\usepackage[ruled,vlined,linesnumbered,noresetcount]{algorithm2e}

\usepackage{hyperref}

\newtheorem{theorem}{Theorem}[section]

\newtheorem{cor}[theorem]{Corollary}

\newtheorem{defn}[theorem]{Definition}
\newtheorem{rmk}{Remark}

\numberwithin{equation}{section}
\def \dv {\mathrm{div}}
\def \d {\mathrm{d}}

\title[Logarithmic convexity of non-symmetric fractional equations]{Logarithmic convexity of non-symmetric time-fractional diffusion equations}

\author{S. E. Chorfi}
\author{L. Maniar}

\address{S. E. Chorfi, L. Maniar, Cadi Ayyad University, Faculty of Sciences Semlalia, LMDP, UMMISCO (IRD-UPMC), B.P. 2390, Marrakesh, Morocco}
\email{s.chorfi@uca.ac.ma, maniar@uca.ac.ma}

\author{M. Yamamoto}
\address{M. Yamamoto, Graduate School of Mathematical Sciences, The University of Tokyo, Komaba, Meguro, Tokyo 153-8914, Japan}
\email{myama@ms.u-tokyo.ac.jp}

\makeatletter 
\@namedef{subjclassname@2020}{%
  \textup{2020} Mathematics Subject Classification}
\makeatother

\subjclass[2020]{35R11, 35R30, 35R25}
 \keywords{Fractional diffusion equation, non-symmetric equation, backward problem, logarithmic convexity, H\"older stability}

\begin{document}

\begin{abstract}
    We consider a class of diffusion equations with the Caputo time-fractional derivative $\partial_t^\alpha u=L u$ subject to the homogeneous Dirichlet boundary conditions. Here, we consider a fractional order $0<\alpha < 1$ and a second-order operator $L$ which is elliptic and non-symmetric. In this paper, we show that the logarithmic convexity extends to this non-symmetric case provided that the drift coefficient is given by a gradient vector field. Next, we perform some numerical experiments to validate the theoretical results in both symmetric and non-symmetric cases. Finally, some conclusions and open problems will be mentioned.
\end{abstract}
\dedicatory{\large Dedicated to the memory of Professor Hammadi Bouslous}
\maketitle

\section{Introduction and main results}
Backward and inverse problems for time-fractional evolution equations have had a growing focus in the recent literature. These equations can model various phenomena, such as slow diffusion and anomalous diffusion, particularly in heterogeneous or porous media; see, for instance, \cite{AG92}. We also refer to \cite[Chapter 10]{Pod99} for other physical applications. In contrast to symmetric equations, non-symmetric evolution equations present some challenges (usually from a spectral viewpoint) and perturbation arguments do not work easily as in the integer case. Hence, few results are available in the literature. Here we aim at pushing forward the investigation of fractional non-symmetric equations within the framework of backward and inverse problems.

Let $n\in \mathbb N$ and $\Omega\subset \mathbb{R}^n$ be a bounded domain with boundary $\partial \Omega$ of class $C^2$. We consider the following backward problem with Dirichlet boundary conditions:
\begin{empheq}[left = \empheqlbrace]{alignat=2}
\begin{aligned}\label{eqe1}
&\partial_{t}^\alpha u(t,x) = L u(t,x), && \qquad\text { in }  (0, T) \times \Omega, \\
& u\rvert_{\partial \Omega} =0, &&\qquad\text { on } (0, T) \times \partial \Omega, \\
& u(T,x)=u_T(x) && \qquad \text{ in } \Omega,
\end{aligned} 
\end{empheq}
posed in the state space $L^2(\Omega)$. Note that other boundary conditions can be considered (e.g., Neumann and Robin). The fractional derivative $\partial_{t}^\alpha u$ stands for the Caputo derivative:
\begin{equation}\label{cap}
    \partial_t^{\alpha} u(t) = \begin{cases}\displaystyle \frac{1}{\Gamma(1-\alpha)} \int^t_0 (t-s)^{-\alpha} \frac{\d }{\d s}u(s) \, \d s, & 0<\alpha<1,\\
    \dfrac{\d }{\d t}u(t), & \alpha=1,
    \end{cases}
\end{equation}
whenever the right-hand side is defined. Here $\Gamma$ denotes the Euler gamma function.

Moreover, $L$ is a non-symmetric elliptic operator of second order given by
\begin{equation}\label{opdef}
L u\, (x) := \sum_{i,j=1}^n \partial_i(a_{ij}\partial_j u)(x) + \sum_{j=1}^n  b_j(x) \partial_j u(x) + p(x) u(x),
\end{equation}
such that
$$
a_{ij} = a_{ji} \in C^1(\overline{\Omega}), \quad b_j \in W^{1,\infty}(\Omega),\quad 1\le i,j \le n, \qquad  p \in L^\infty(\Omega),
$$
and there exists a constant $\kappa>0$ such that 
$$
\sum_{i,j=1}^n a_{ij}(x)\xi_i\xi_j \ge \kappa \sum_{j=1}^n \xi_j^2,
\quad x \in \overline{\Omega}, \; (\xi_1, ..., \xi_n) \in \mathbb{R}^n.
$$
The unit outward normal vector field to $\partial \Omega$ is denoted by $\nu$ and the conormal derivative with respect to $a=(a_{ij})$ by
$$\partial_{\nu}^{a} u=\sum_{i, j=1}^{n} a_{i j} \nu_{j} (\partial_{i} u)_{|\partial \Omega}.$$
The domain of the operator $L$ is given by
\begin{equation} \label{domop}
    D(L)= H^2(\Omega) \cap H^1_0(\Omega).
\end{equation}
Set $\mathcal{A}=(a_{ij})_{ij}$ for the diffusion matrix and $\mathcal{B}=(b_1, b_2, \ldots, b_n)$ for the drift term. Then the operator $L$ defined by \eqref{opdef} can be written as
$$Lu\, (x) =\mathrm{div}(\mathcal{A}(x)\nabla u(x)) + \mathcal{B}(x)\cdot \nabla u(x) + p(x) u(x).$$
Next, we introduce the main assumption on the drift term:
\begin{equation*}\label{H}
\textbf{(H)} \quad \text{There exists a function } b\in W^{2,\infty}(\Omega) \text{ such that } \mathcal{B}=\mathcal{A}\nabla b.    
\end{equation*}
Assumption \textbf{(H)} is satisfied for the isotropic case $\mathcal{A}=I_n$ (the identity matrix) and gradient drift $\mathcal{B}=\nabla b$. We emphasize that this special case of Assumption \textbf{(H)} has been widely considered in the literature. It appears for instance in differential topology and theoretical physics (Helffer-Sjöstrand theory \cite{HS85}) as it relates to Witten Laplacian. Furthermore, such an assumption has been useful in resolving some uniform controllability problems for parabolic equations with vanishing viscosity; see \cite{LL21} and its bibliography.

Let us state the main result of the logarithmic convexity of the non-symmetric equation \eqref{eqe1}.
\begin{theorem} \label{thm1}
Assume that Assumption \textbf{(H)} is fulfilled. Let $u_T\in L^2(\Omega)$ and let $u\in C([0,T];L^2(\Omega))$ be a solution to \eqref{eqe1}. Then there exists a constant $\kappa=\kappa(\mathcal{A}, b, p,\alpha, T) \ge 1$ such that
\begin{equation}\label{lceq1}
\|u(t,\cdot)\|_{L^2(\Omega)} \le \kappa \mathrm{e}^{\|b\|_\infty} \|u(0,\cdot)\|^{1-\frac{t}{T}}_{L^2(\Omega)} \|u(T,\cdot)\|^{\frac{t}{T}}_{L^2(\Omega)}, \qquad 0\le t \le T.
\end{equation}
\end{theorem}

As a direct consequence, we infer the following backward uniqueness property.
\begin{cor}
Under the assumptions of Theorem \ref{thm1}, the backward uniqueness for \eqref{eqe1} holds, i.e., if $u(T, \cdot)= 0$, then $u(t,\cdot)= 0$ for all $t\in [0,T]$.
\end{cor}

\begin{rmk}
Assumption \textbf{(H)} is sufficient but not necessary for the logarithmic convexity estimate; see Section \ref{sec4}, Example 3.
\end{rmk}

The logarithmic convexity method is a well-known approach that has been extensively applied to prove backward uniqueness and conditional stability for inverse and ill-posed problems. We refer to the survey paper \cite{Car99} and the cited bibliography. This method is beneficial for reconstructing the initial data and studying backward problems \cite{Car13}. Interested readers can find more details in the books \cite{Pa'75, AS'97}. We also refer to \cite{KP'60, Mi75} for abstract results on analytic semigroups. In particular, as for the backward uniqueness, see \cite{Gh'86} and the references therein. It should be pointed out that the logarithmic convexity method provides a very explicit estimate even for non-symmetric cases, although its applicability is limited.

Many works are available for backward problems for time-fractional evolution equations in the symmetric case. For instance, in \cite{WWZ13}, Tikhonov regularization and conditional stability have been proposed based on a Fredholm integral equation for the numerical reconstruction of initial states. The article \cite{Ha19} obtains optimal stability estimates and a regularization scheme based on a non-local boundary value method with parameter choice rules for a linear equation. We refer to \cite{THNZ19} for a nonlinear equation where the existence, uniqueness, and regularization of a local solution have been investigated. It should be emphasized that most results rely on eigenfunctions expansion and the properties of Mittag-Leffler functions.

The non-symmetric case still needs to be explored regarding backward uniqueness and logarithmic convexity (see Section \ref{sec5} for more details). Indeed, we are only aware of \cite{FLY'20}, where the authors prove the well-posedness of a backward problem for time-fractional diffusion equations with a non-symmetric operator. The proofs draw on a perturbation argument and the completeness of generalized eigenfunctions. We also mention the recent paper \cite{JLPY23} for a boundary unique continuation result applied to an inverse source problem. The proof is based on the Laplace transform and the spectral decomposition. For numerical aspects, we refer to the recent paper \cite{YWL22} for a one-dimensional advection-dispersion equation where a quasi-boundary regularization method has been proposed.

The structure of the paper will be as follows: in Section \ref{sec2}, we introduce some preliminary results that will be useful for the sequel. In particular, we extend the definition of logarithmic convexity to include time-fractional evolution equations. Section \ref{sec3} is devoted to the proof of Theorem \ref{thm1}. In Section \ref{sec4}, we provide some numerical experiments to validate our theoretical results in both symmetric and non-symmetric cases. Finally, we conclude this paper with some conclusions and open problems.

\section{Preliminary results \label{sec2}}
We start this section by recalling some preliminaries that will be used later throughout this paper.

Let $0 < \alpha \le 1$ and $T>0$ be fixed. Let $X$ be a Banach space, and let $\|\cdot\|$ be its associated norm. We consider the abstract fractional Cauchy problem
\begin{empheq}[left = \empheqlbrace]{alignat=2}
\begin{aligned}\label{eq1}
&\partial_{t}^\alpha u(t) = A u(t), && \qquad  t\in (0, T), \\
& u(0)=u_0,
\end{aligned}
\end{empheq}
where $A: D(A) \subset X \rightarrow X$ generates a $C_0$-semigroup $\left(S(t)\right)_{t\ge0}$ in $X$. The solution of \eqref{eq1} is given by $u(t)=S_\alpha(t)u_0$, with the solution operator given by
\begin{equation}\label{repeq}
    S_\alpha(t)v=\int_0^{\infty} \Phi_\alpha(s) S\left(s t^\alpha \right)v \, \d s, \quad v\in X,
\end{equation}
and
$$\Phi_\alpha(z):=\sum_{n=0}^{\infty} \frac{(-z)^n}{n !\,  \Gamma(-\alpha n+1-\alpha)}, \quad z \in \mathbb{C},$$
is the Wright function which is a probability density function:
$$\Phi_\alpha(t) \geq 0, \quad t>0, \quad \int_0^{\infty} \Phi_\alpha(t)\, \d t=1 .$$
We denote by $E_{\alpha}(z)$ the Mittag-Leffler function
$$
E_{\alpha}(z) = \sum_{k=0}^{\infty} \frac{z^k}{\Gamma(\alpha k + 1)}, \qquad z\in \mathbb{C},
$$
which is an entire function in $\mathbb{C}$, see, e.g., \cite{Gor20}. Note that the two functions are related by the formula
$$E_{\alpha}(z)=\int_0^\infty \Phi_\alpha(t)\, \mathrm{e}^{zt}\, \d t, \qquad z\in \mathbb{C},\; 0<\alpha<1.$$
We refer to the thesis \cite{Baj'01} for a detailed exposition.

Inspired by \cite[Definition 3.1]{CM'23} for the integer case $\alpha=1$, we introduce the following general definition.
\begin{defn}
Let $0 < \alpha \le 1$. We say that the solution operator $(S_\alpha(t))_{t\ge 0}$ to \eqref{eq1} satisfies a logarithmic convexity estimate for $T >0$ if there exist a constant $\kappa=\kappa(T,\alpha) \ge 1$ and a continuous function $w \colon (0,T) \rightarrow (0,1)$, $w(0)=0$ and $w(T)=1$, such that the following estimate
\begin{equation}\label{elc}
\left\|S_\alpha(t) u_0\right\| \le \kappa \|u_0\|^{1-w(t)} \left\|S_\alpha(T)u_0\right\|^{w(t)}
\end{equation}
holds for all $t\in [0,T]$ and all $u_0\in X$.
\end{defn}
We collect a few facts about logarithmic convexity.
\begin{itemize}
    \item[$\bullet$] \textbf{Integer case $\alpha=1$:} logarithmic convexity holds for symmetric operators with $\kappa=1$ and $w(t)=\dfrac{t}{T}$ (see e.g. \cite[Section 2]{GaT'11}). This results from the fact that $t\mapsto \|u(t)\|$ is a log-convex function (hence the name logarithmic convexity). If the operator is subordinated to its symmetric part, a logarithmic convexity result is established in \cite[Theorem 3.1.3]{Is'17} with $w(t)=\dfrac{1-\mathrm{e}^{-ct}}{1-\mathrm{e}^{-cT}}$, $c>0$ being constant. More generally, for analytic semigroups, a more general function $w(t)$ (given by the so-called harmonic measure) satisfies \eqref{elc}. More details can be found in \cite{ACM'23}. It should be emphasized that some logarithmic convexity estimates have recently been obtained for non-analytic cases in \cite{CM'23, CM2'23} with $w(t)=c_T\dfrac{t}{T}$, $c_T\in (0,1]$ is constant.
    \item[$\bullet$] \textbf{Fractional case $0<\alpha<1$:} only a few results are known. For instance, logarithmic convexity holds for symmetric operators such as \eqref{opdef} without drift term, i.e., $\mathcal{B}=0$. It has recently been proven in \cite{CMY'22} that \eqref{elc} holds in this case with $w(t)=\dfrac{t}{T}$. The general case of non-symmetric operators remains an open problem. This motivates our present work.
\end{itemize}
For future use, we recall the following result from \cite{CMY'22} on the symmetric case. Let $0 < \alpha \le 1$ and $X=H$ be a separable Hilbert space.
\begin{theorem} \label{thm2}
Assume that the operator $A$ is self-adjoint, bounded above and has compact resolvent. Then the associated solution $u$ to \eqref{eq1} satisfies the logarithmic convexity estimate for $T >0$ with $w(t)=\dfrac{t}{T}$, $t\in [0,T]$. That is, there exists a constant $\kappa\ge 1$ such that
\begin{equation}\label{lceq2}
\|u(t)\| \le \kappa \|u_0\|^{1-\frac{t}{T}} \|u(T)\|^{\frac{t}{T}}, \qquad 
0\le t \le T
\end{equation}
for all $u_0 \in H$.
\end{theorem}

\section{Proof of the logarithmic convexity estimate \label{sec3}}
In this section, we prove the main result of Theorem \ref{thm1}. The key idea is to transform the equation \eqref{eqe1} into a symmetric equation of the same kind, thanks to Assumption \textbf{(H)}.

\begin{proof}[Proof of Theorem \ref{thm1}]
Let us make the change of variable $$v(t,x)=\mathrm{e}^{\frac{b}{2}} u(t,x),$$ where $u(t,x)$ is a solution to \eqref{eqe1}. Then, by a simple calculation, the following identities hold
\begin{align*}
    \partial_t^\alpha u(t,x)&= \mathrm{e}^{-\frac{b}{2}} \partial_t^\alpha v(t,x),\\
    \nabla u &= \mathrm{e}^{-\frac{b}{2}} \left(-\frac{1}{2}v \nabla b + \nabla v\right),\\
    \dv(\mathcal{A}\nabla u)&=\mathrm{e}^{-\frac{b}{2}} \left(\dv(\mathcal{A}\nabla v)-\mathcal{A}\nabla b\cdot \nabla v+\left(\frac{1}{4} \mathcal{A}\nabla b\cdot \nabla b-\frac{1}{2} \dv(\mathcal{A}\nabla b)\right)v\right),\\
    \mathcal{B}\cdot \nabla u + pu &=\mathrm{e}^{-\frac{b}{2}} \left(\left(-\frac{1}{2} \mathcal{A}\nabla b\cdot \nabla b +p\right) v + \mathcal{A}\nabla b\cdot \nabla v \right).
\end{align*}
Therefore, we obtain that $v$ is a solution to the backward problem
\begin{empheq}[left = \empheqlbrace]{alignat=2}
\begin{aligned}\label{eqv1}
&\partial_{t}^\alpha v(t,x) = L_0 v(t,x), && \qquad\text { in }  (0, T) \times \Omega, \\
& v\rvert_{\partial \Omega} =0, &&\qquad\text { on } (0, T) \times \partial \Omega, \\
& v(T,x)=v_T(x) && \qquad \text{ in } \Omega,
\end{aligned} 
\end{empheq}
where $v_T =\mathrm{e}^{\frac{b}{2}} u_T$, $v_0 =\mathrm{e}^{\frac{b}{2}} u_0$, and the operator $L_0$ is given by
\begin{equation}\label{L0}
    L_0 v\, (x) =\mathrm{div}(\mathcal{A}(x)\nabla v(x)) + q(x) v(x),
\end{equation}
with
$$q(x)=p(x)-\frac{1}{2} \dv(\mathcal{A}(x)\nabla b(x)) -\frac{1}{4} \mathcal{A}(x)\nabla b(x)\cdot \nabla b(x), \qquad x\in \Omega.$$
Since $\mathcal{A}$ is of class $C^1$ in $\overline{\Omega}$ and $b\in W^{2,\infty}(\Omega)$ (by Assumption \textbf{(H)}), then $q\in L^\infty(\Omega)$. With the domain
$$D(L_0)= H^2(\Omega) \cap H^1_0(\Omega),$$
it is well-known that $L_0$ is a self-adjoint operator that is bounded above with compact resolvent. Then by applying Theorem \ref{thm2} to the equation \eqref{eqv1}, there exists $\kappa\ge 1$ such that
\begin{equation}\label{lcv}
    \|v(t,\cdot)\|_{L^2(\Omega)} \le \kappa \|v_0\|_{L^2(\Omega)}^{1-\frac{t}{T}} \|v_T\|_{L^2(\Omega)}^{\frac{t}{T}}, \qquad 
0\le t \le T.
\end{equation}
On the other hand, by $b\in L^\infty(\Omega)$, we obtain
$$\|u(t,\cdot)\|_{L^2(\Omega)} \le \mathrm{e}^{\frac{1}{2}\|b\|_\infty} \|v(t,\cdot)\|_{L^2(\Omega)}, \qquad 
0\le t \le T.$$
In particular,
$$\|v_0\|_{L^2(\Omega)} \le \mathrm{e}^{\frac{1}{2}\|b\|_\infty} \|u_0\|_{L^2(\Omega)}, \qquad \|v_T\|_{L^2(\Omega)} \le \mathrm{e}^{\frac{1}{2}\|b\|_\infty} \|u_T\|_{L^2(\Omega)}.$$
Plugging the above inequalities in \eqref{lcv}, we get
$$
\|u(t,\cdot)\|_{L^2(\Omega)} \le \kappa \mathrm{e}^{\|b\|_\infty}\|u_0\|_{L^2(\Omega)}^{1-\frac{t}{T}} \|u_T\|_{L^2(\Omega)}^{\frac{t}{T}}, \qquad 
0\le t \le T.
$$
Thus, the proof is completed.
\end{proof}

\section{Numerical experiments \label{sec4}}
In this section, we use the \texttt{Wolfram} language to validate the theoretical results discussed above. In particular, we use the function \texttt{DSolve} for solving a simple PDE and the recent function \texttt{CaputoD} for the Caputo fractional derivative. Note that the predefined function \texttt{NDSolve} in the last update of \texttt{Wolfram} language cannot solve fractional PDEs numerically.

To approximate the Caputo fractional derivative and solve the corresponding PDEs, we use the classical $L1$ scheme; see \cite[Section 8.2]{OS74} and \cite{Ga14}. Let us briefly recall this method. For the temporal interval $[0,T]$, we denote by $\Delta t=\frac{T}{N}$ the temporal step size of the grid $t_k=k \Delta t, \; k= 0,1\ldots, N$ ($N\in \mathbb N$). By definition, we have
\begin{align*}
    \left.\partial_t^\alpha u(t)\right|_{t=t_k}&=\frac{1}{\Gamma(1-\alpha)} \int_0^{t_k} \left(t_k-s\right)^{-\alpha} u^{\prime}(s)\, \mathrm{d} s\\
    &=\frac{1}{\Gamma(1-\alpha)} \sum_{j=1}^k \int_{t_{j-1}}^{t_j} \left(t_k-s\right)^{-\alpha} u^{\prime}(s)\, \mathrm{d} s .
\end{align*}
Then by using a piecewise linear interpolation of $u(t)$ in each subinterval $[t_{j-1},t_j]$ for $1\le j \le k$, we obtain the $L1$ approximating formula for $0< \alpha <1$,
\begin{equation}\label{L1}
    \left.D_t^\alpha u(t)\right|_{t=t_k} =\frac{(\Delta t)^{-\alpha}}{\Gamma(2-\alpha)}\left(u\left(t_k\right)-\sum_{j=1}^{k-1}\left(a_{k-j-1}^{(\alpha)}-a_{k-j}^{(\alpha)}\right) u\left(t_j\right)-a_{k-1}^{(\alpha)} u\left(t_0\right)\right),
\end{equation}
where $a_{j}^{(\alpha)}:=(j+1)^{1-\alpha}-j^{1-\alpha}, \; 0\le j\le k-1.$ Note that the $L1$ formula \eqref{L1} has an accuracy of order $2-\alpha$, and other higher-order schemes can also be used; see, e.g., \cite{MM20}.

Now using the formula \eqref{L1} to approximate the time-fractional derivative and a finite difference scheme for the spatial derivative, we can discretize the fractional PDE
\begin{equation}\label{eq0}
    \begin{cases}
    \partial^{\alpha}_t v(t,x)= v_{xx}(t,x)+q(x)v(t,x), & (t,x)\in (0,T) 
    \times (0,1),\\
    v(t,0)=v(t,1)=0, & t\in (0,T),\\
    v(0, x)=v_0(x), & x\in (0,1).
    \end{cases}
\end{equation}
Indeed, for the spatial interval $[0, 1]$, we consider the grid $x_i = i \Delta x, \; i = 0,1,\ldots, M$ ($M\in \mathbb N$), where $\Delta x=\frac{1}{M}$ denotes the spatial mesh size. Then we set $v_i:=v(x_i)$ and use the centered difference approximation for $v_{xx}$:
$$\delta_x^2 v_i = \frac{v_{i-1}-2 v_i + v_{i+1}}{(\Delta x)^2}, \quad i=1,\ldots, M-1.$$ Then the equation \eqref{eq0} is approximated by
\begin{equation}\label{eq00}
    \begin{cases}
    D^{\alpha}_t v_i^k = \delta_x^2 v_i^k + q(x_i)v_i^k, & 1\le i \le M-1, \; 1\le k\le N,\\
    v_0^k=v_M^k=0, & 1\le k\le N,\\
    v_i^0=v_0(x_i), & 0\le i \le M.
    \end{cases}
\end{equation}
Next, we fix $N=20$ and $M=80$ in the following examples.

\subsection*{Example 1: a symmetric case} \label{Ex1}
To illustrate the difficulty in non-symmetric equations, we start with the following symmetric equation
\begin{equation}\label{ex1}
    \begin{cases}
    \partial^{\alpha}_t u(t,x)= u_{xx}(t,x), & (t,x)\in (0,0.02) 
    \times (0,1),\\
    u(t,0)=u(t,1)=0, & t\in (0,0.02),\\
    u(0, x)=\sin(\pi x), & x\in (0,1).
    \end{cases}
\end{equation}
This equation can be resolved by the Laplace transform technique as follows. Denote $\displaystyle\mathcal{L}\{f(t)\}(s):=\int_0^\infty \mathrm{e}^{-st}f(t) \d t$ the Laplace transform (with respect to $t$) of $f(t)$. Applying the Laplace transform to \eqref{ex1} and solving the corresponding ordinary differential equation in $x$, we obtain $\mathcal{L}\{u(t,x)\}(s)=\dfrac{s^{\alpha-1}}{s^\alpha+\pi^2} \sin(\pi x)$. Therefore, by the inverse Laplace transform, the solution of equation \eqref{ex1} is given by the following formula
$$u_\alpha(t,x)=E_\alpha\left(-\pi^2 t^\alpha\right) \sin(\pi x), \quad t\in (0,0.02), \; x\in (0,1).$$
Next, we plot the solution $u_\alpha(t,x)$ for $\alpha = 0.1$.
\begin{figure}[H]
\centering
\includegraphics[scale=0.5]{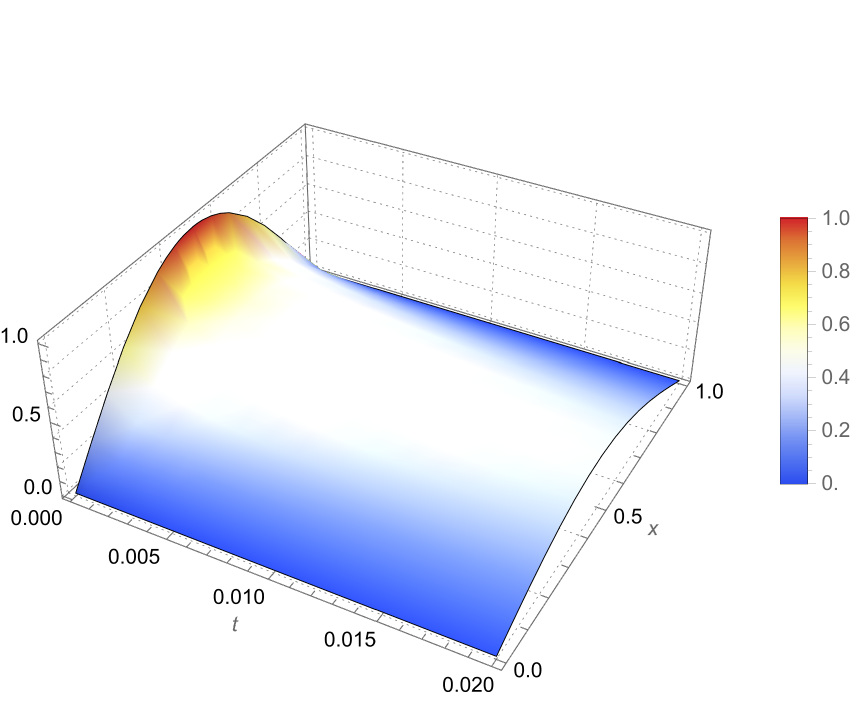}
\caption{The solution $u_\alpha(t,x)$ for $\alpha = 0.1$ in Example 1.}
\label{fig00}
\end{figure}
Now we plot $\log\|u_\alpha(t,\cdot)\|_{L^2(0,1)}$ for $\alpha = 0.1, \;0.3,\; 0.5$.
\begin{figure}[H]
\centering
\includegraphics[scale=0.5]{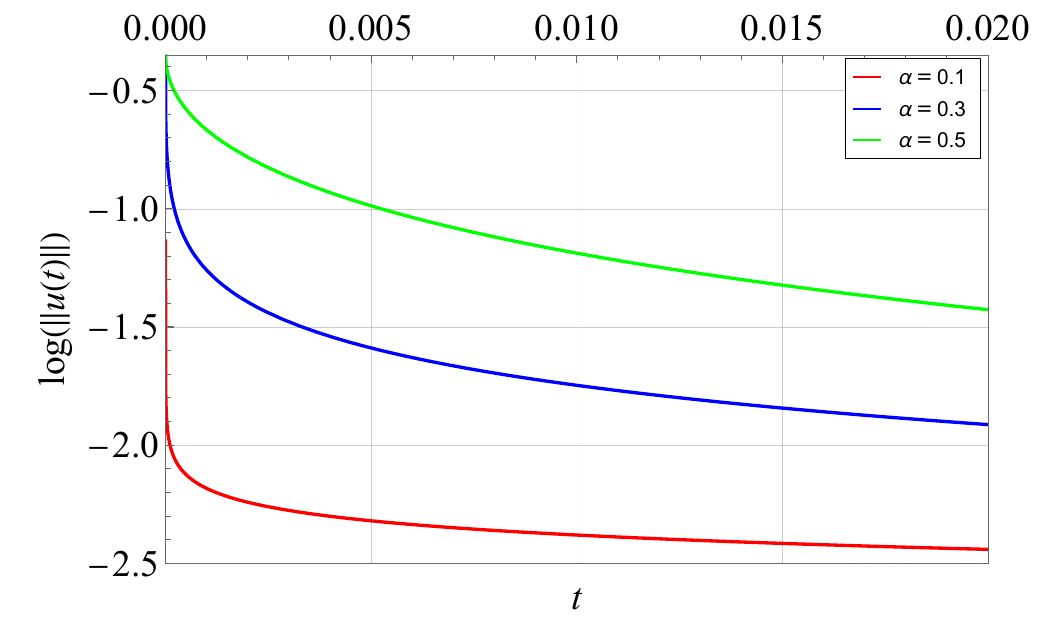}
\caption{$\log\|u_\alpha(t,\cdot)\|_{L^2(0,1)}$ for $\alpha = 0.1, \;0.3,\; 0.5$ in Example 1.}
\label{fig0}
\end{figure}
Figure \ref{fig0} shows that, for $\alpha = 0.1, \;0.3,\; 0.5$, the function $t\mapsto\|u_\alpha(t,\cdot)\|_{L^2(0,1)}$ is log-convex. This is compatible with the theoretical results established in the symmetric case; see \cite{CMY'22}. However, the results therein do not cover non-symmetric equations.

\subsection*{Example 2: a non-symmetric case with Assumption \textbf{(H)}}\label{Ex2}
In this example, we perturb the equation \eqref{ex1} of the previous example and consider instead the following non-symmetric equation
\begin{equation}\label{ex2}
    \begin{cases}
    \partial^{\alpha}_t u(t,x)= u_{xx}(t,x)+u_{x}(t,x), & (t,x)\in (0,0.02) \times (0,1),\\
    u(t,0)=u(t,1)=0, & t\in (0,0.02),\\
    u(0, x)=\sin(\pi x), & x\in (0,1).
    \end{cases}
\end{equation}
It should be emphasized that it is difficult to obtain an explicit formula for the solution to \eqref{ex2} similarly to Example 1. Indeed, by using the symmetrization process presented in Section \ref{sec3}, the function $v(t,x)=\mathrm{e}^{\frac{x}{2}} u(t,x)$ satisfies
\begin{equation}\label{ex2b}
    \begin{cases}
    \partial^{\alpha}_t v(t,x)= v_{xx}(t,x)-\frac{1}{4} v(t,x), & (t,x)\in (0,0.02) \times (0,1),\\
    v(t,0)=v(t,1)=0, & t\in (0,0.02),\\
    v(0, x)=\mathrm{e}^{\frac{x}{2}} \sin(\pi x), & x\in (0,1).
    \end{cases}
\end{equation}
However, an analytical formula for the solution to \eqref{ex2b} is still difficult due to the additional term $\frac{1}{4} v(t,x)$ and the exponential that appears in the initial datum $v(0,x)$. Furthermore, even the Laplace transform technique fails in this context. Indeed, if we apply the Laplace transform to \eqref{ex2b} and solve the corresponding ordinary differential equation in $x$, we obtain a very complicated formula that is difficult to invert. Thus, we only solve this equation numerically by the finite difference method presented earlier.

Next, we plot the solution $u_\alpha(t,x)=\mathrm{e}^{-\frac{x}{2}} v(t,x)$ to equation \eqref{ex2} for $\alpha = 0.1$.
\begin{figure}[H]
\centering
\includegraphics[scale=0.5]{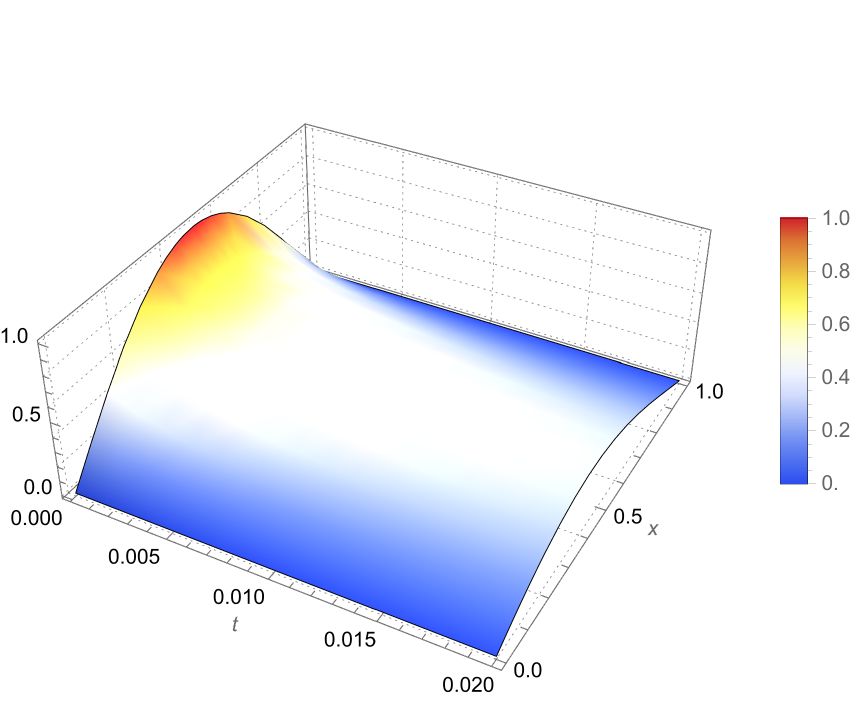}
\caption{The solution $u_\alpha(t,x)$ for $\alpha = 0.1$ in Example 2.}
\label{fig11}
\end{figure}
Now we plot $\log\|u_\alpha(t,\cdot)\|_{L^2(0,1)}$ for $\alpha = 0.1, \;0.3,\; 0.5$.
\begin{figure}[H]
\centering
\includegraphics[scale=0.5]{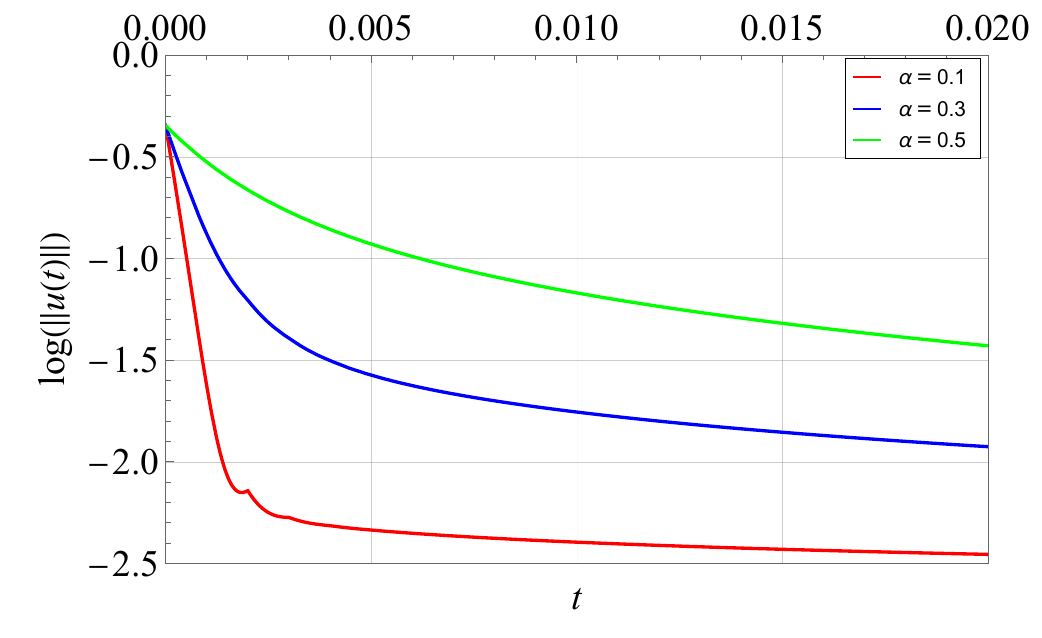}
\caption{$\log\|u_\alpha(t,\cdot)\|_{L^2(0,1)}$ for $\alpha = 0.1, \;0.3,\; 0.5$ in Example 2.}
\label{fig1}
\end{figure}
Figure \ref{fig1} shows that, for $\alpha = 0.1, \;0.3,\; 0.5$, the function $t\mapsto\|u_\alpha(t,\cdot)\|_{L^2(0,1)}$ is also log-convex as in the symmetric case.

In our main result, we have proven that Assumption \textbf{(H)} is sufficient for the logarithmic convexity estimate to hold. The following numerical example shows that this assumption is not necessary for logarithmic convexity.

\subsection*{Example 3: a non-symmetric case without Assumption \textbf{(H)}}\label{Ex3}
Now we consider a non-symmetric equation with a non-smooth drift coefficient
\begin{equation}\label{ex3}
    \begin{cases}
    \partial^{\alpha}_t u(t,x)= u_{xx}(t,x) + \theta\left(x-\frac{1}{2}\right)u_{x}(t,x), & (t,x)\in (0,0.02) \times (0,1),\\
    u(t,0)=u(t,1)=0, & t\in (0,0.02),\\
    u(0, x)=\sin(\pi x), & x\in (0,1),
    \end{cases}
\end{equation}
where $\theta(x):=\begin{cases}1,&x\geq 0,\\0,&x<0,\end{cases}$ denotes the Heaviside function. Note that the drift coefficients are given by $\mathcal{B}(x):=\theta\left(x-\frac{1}{2}\right)\notin W^{1,\infty}(0,1)$ and $b(x):=\left(x-\frac{1}{2}\right)\theta\left(x-\frac{1}{2}\right)+c \notin W^{2,\infty}(0,1),$ $c$ is constant. Therefore, this example violates Assumption \textbf{(H)}, and one cannot use the symmetrization process as in Example 2.

We solve this equation numerically by the finite differences as before. Then, we plot the solution $u_\alpha(t,x)$ to equation \eqref{ex3} for $\alpha = 0.1$.
\begin{figure}[H]
\centering
\includegraphics[scale=0.5]{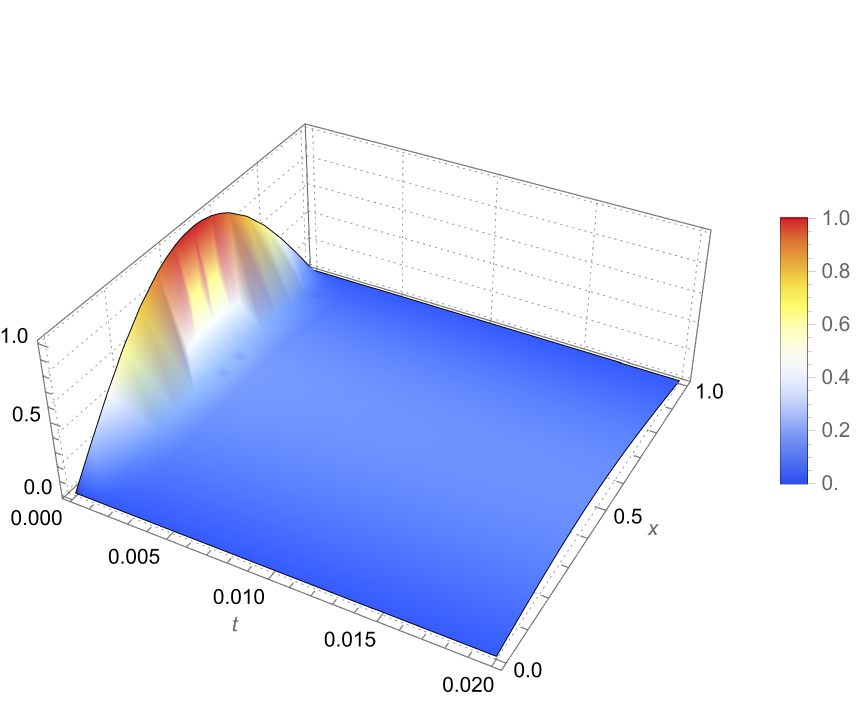}
\caption{The solution $u_\alpha(t,x)$ for $\alpha = 0.1$ in Example 3.}
\label{fig31}
\end{figure}
Now we plot $\log\|u_\alpha(t,\cdot)\|_{L^2(0,1)}$ for $\alpha = 0.1, \;0.3,\; 0.5$.
\begin{figure}[H]
\centering
\includegraphics[scale=0.5]{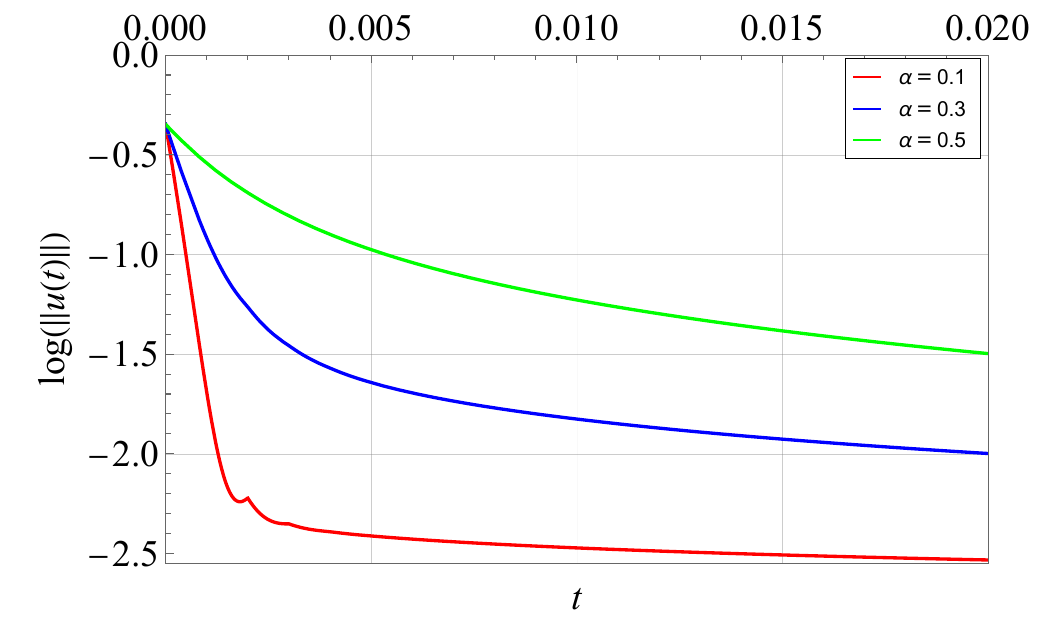}
\caption{$\log\|u_\alpha(t,\cdot)\|_{L^2(0,1)}$ for $\alpha = 0.1, \;0.3,\; 0.5$ in Example 3.}
\label{fig3}
\end{figure}
Figure \ref{fig3} shows that, for $\alpha = 0.1, \;0.3,\; 0.5$, the function $t\mapsto\|u_\alpha(t,\cdot)\|_{L^2(0,1)}$ is also log-convex. Note that Theorem \ref{thm1} does not apply to this case. However, several numerical tests with different initial data show that logarithmic convexity still holds in this case, even if Assumption \textbf{(H)} is not fulfilled.

\section{Final comments and open problems}\label{sec5}
This section will address some concluding comments and open problems that deserve further investigation.

\subsection{Logarithmic convexity without Assumption \textbf{(H)}}
We have shown that Assumption \textbf{(H)} is sufficient for the logarithmic convexity estimate to hold for the solutions of non-symmetric equations. We have also seen in Example 3 that this assumption is not necessary. It is then natural to ask how this assumption can be dropped (or weakened) in the proof. Indeed, in the integer case $\alpha=1$, the logarithmic convexity estimate holds for a general non-symmetric operator $L$ defined by \eqref{opdef} without assuming \textbf{(H)}, see \cite[Example 3.1.6]{Is'17} and \cite[Lemma A.4]{ACM'22} for a general abstract result. We expect the same general result for the fractional case $0<\alpha<1$.

\subsection{Backward uniqueness for analytic semigroups}
First, recall that the backward uniqueness is a weaker property implied by logarithmic convexity. In the integer case $\alpha=1$, it is known that the solution of the Cauchy problem corresponding to an analytic semigroup satisfies the backward uniqueness property. This fact results from analyticity and the semigroup law \cite{KP'60, Mi75}. However, in the fractional case $0<\alpha<1$, although the analyticity of the solutions is preserved as shown in \cite{Baj'01, LHY20}, there is no standard analog of the semigroup law. In the fractional case with a non-symmetric operator $L$, we know that the backward uniqueness holds given the results in \cite{FLY'20}. We expect an affirmative answer to this question in a general framework of Banach spaces, but the proof is still an open problem to be resolved.

\end{document}